\documentclass[pagesize]{scrartcl}
\usepackage{latexsym} 
\usepackage[intlimits]{amsmath}
\usepackage{amsthm}
\usepackage{amsfonts}
\usepackage{amssymb}
\usepackage{amsxtra}
\usepackage{amscd}
\usepackage{ifthen}
\usepackage{graphicx}
\usepackage{enumitem}
\usepackage{mathrsfs}
\usepackage{mathtools}
\usepackage[pagebackref=true]{hyperref}

\hypersetup{
pdfauthor={Toshiaki Hishida, Mads Kyed},
pdftitle={On the asymptotic structure of steady Stokes and Navier-Stokes flows around a rotating two-dimensional body},
breaklinks=true,
colorlinks=true,
linkcolor=blue,
citecolor=blue,
urlcolor=blue,
filecolor=blue,
}

\numberwithin{equation}{section} 
\setkomafont{title}{\normalfont}

\newenvironment{pdeq}{ \left\{ \begin{aligned}}{\end{aligned}\right.}

\newcommand{\np}[1]{(#1)}
\newcommand{\nb}[1]{[#1]}
\newcommand{\bp}[1]{\big(#1\big)}
\newcommand{\bcp}[1]{\big\{#1\big\}}
\newcommand{\bb}[1]{\big[#1\big]}
\newcommand{\Bp}[1]{\bigg(#1\bigg)}

\newcommand{\Bb}[1]{\bigg[#1\bigg]}
\newcommand{\calb}{{\mathcal B}}

\newcommand{\calt}{{\mathcal T}}

\newcommand{\R}{\mathbb{R}}
\newcommand{\C}{\mathbb{C}}
\newcommand{\Z}{\mathbb{Z}}

\newcommand{\convolutionpergroup}{*}
\newcommand{\convolutioneuclidean}{*_{\R^2}}
\DeclareMathOperator{\e}{e}

\DeclareMathOperator{\Div}{div}

\DeclareMathOperator{\trace}{Tr}

\DeclareMathOperator{\realpart}{Re}

\newcommand{\ra}{\rightarrow}

\newcommand{\set}[1]{\ensuremath{\{#1\}}}

\newcommand{\setc}[2]{\ensuremath{\{#1\ \lvert\ #2\}}}
\newcommand{\setcl}[2]{\ensuremath{\bigl\{#1\ \lvert\ #2\bigr\}}}

\newcommand{\quotientmap}{\pi}
\newcommand{\grp}{{\torus_\per\times\R^2}}
\newcommand{\grpper}{{\torus_\per\times\R^2}}

\newcommand{\dualgrp}{\Z\times\R^2}

\newcommand{\torus}{{\mathbb T}}
\newcommand{\torusper}{{{\mathbb T}_\per}}
\newcommand{\torusone}{{{\mathbb T}_1}}

\newcommand{\transpose}{\top}

\newcommand{\idmatrix}{I}
\newcommand{\grad}{\nabla}

\newcommand{\pdax}{\partial_{x_j}}

\newcommand{\dx}{{\mathrm d}x}

\newcommand{\ds}{{\mathrm d}s}
\newcommand{\dt}{{\mathrm d}t}

\newcommand{\dy}{{\mathrm d}y}

\newcommand{\dxi}{{\mathrm d}\xi}

\newcommand{\SR}{\mathscr{S}}

\newcommand{\TDR}{\mathscr{S^\prime}}

\newcommand{\FT}{\mathscr{F}}
\newcommand{\iFT}{\mathscr{F}^{-1}}

\newcommand{\torusmultiplier}{m_x}
\newcommand{\reallinemultiplier}{M_x}

\newcommand{\hgamma}{h_\gamma}
\newcommand{\ggamma}{g_\gamma}
\newcommand{\besselK}{K}
\newcommand{\besseleone}{e_1}
\newcommand{\besseletwo}{e_2}
\newcommand{\besselej}{e_j}

\newcommand{\norm}[1]{\lVert#1\rVert}

\newcommand{\snorm}[1]{{\lvert #1 \rvert}}
\newcommand{\snorml}[1]{{\bigl\lvert #1 \big\rvert}}

\newcommand{\CR}[1]{C^{#1}}  
\newcommand{\LR}[1]{L^{#1}}

\newcommand{\CRi}{\CR \infty}

\newcommand{\wspace}[1]{X_{#1}}
\newcommand{\wspacelog}[1]{X^{\textrm{log}}_{#1}}
\newcommand{\nsnonlinb}[2]{#1\cdot\grad #2}

\newcommand{\fundsolstokesvel}{\varGamma^{\textnormal{\tiny{S}}}}

\newcommand{\fundsoltpvelper}{\varGamma^{\per}}
\newcommand{\fundsoltppresper}{\gamma^{\per}}
\newcommand{\fundsolcomplper}{\varGamma^{\per,\bot}}
\newcommand{\fundsolcomplone}{\varGamma^{1,\bot}}
\newcommand{\fundsolssrk}{\varGamma^{k}}
\newcommand{\fundsolssreta}{\varGamma^{\eta}}

\newcommand{\tin}{\text{in }}

\newcommand{\tas}{\text{as }}
\newcommand{\half}{\frac{1}{2}}

\renewcommand{\epsilon}{\varepsilon}
\renewcommand{\phi}{\varphi}

\newcommand{\per}{\calt}

\newcommand{\perf}{\frac{2\pi}{\per}}
\newcommand{\perone}{2\pi}

\newcommand{\rotmatrix}{Q}
\newcommand{\bigo}{O}

\newcommand{\cutoff}{\chi}

\newcommand{\angvel}{a}

\newcommand{\onedist}{1}
\newcommand{\xrot}{x^\bot}

\newcommand{\body}{\calb}
\newcommand{\entireP}{P}
\newcommand{\entireQ}{Q}

\newcommand{\intmatrix}[1]{M^{#1}}

\newcommand{\newCCtr}[2][d]{
\newcounter{#2}\setcounter{#2}{0}
\expandafter\xdef\csname kyedtheconst#2\endcsname{#1}
}
\newcommand{\Cc}[2][nolabel]{
\stepcounter{#2}
\expandafter\ensuremath{\csname kyedtheconst#2\endcsname_{\arabic{#2}}}
\ifthenelse{\equal{#1}{nolabel}}
{}
{\expandafter\xdef\csname kyedconst#1\endcsname
{\expandafter\ensuremath{\csname kyedtheconst#2\endcsname_{\arabic{#2}}}}}
}
\newcommand{\Ccn}[2][nolabel]{
\expandafter\ensuremath{\csname kyedtheconst#2\endcsname}
\ifthenelse{\equal{#1}{nolabel}}
{}
{\expandafter\xdef\csname kyedconst#1\endcsname
{\expandafter\ensuremath{\csname kyedtheconst#2\endcsname}}}
}
\newcommand{\CcSetCtr}[2]{
\setcounter{#1}{#2}
}
\newcommand{\Cclast}[1]{
\expandafter\ensuremath{\csname kyedtheconst#1\endcsname_{\arabic{#1}}}
}
\newcommand{\Ccllast}[1]{
\addtocounter{#1}{-1}
\expandafter\ensuremath{\csname kyedtheconst#1\endcsname_{\arabic{#1}}}
\addtocounter{#1}{1}
}
\newcommand{\const}[1]{
\expandafter{\ifcsname kyedconst#1\endcsname
  \csname kyedconst#1\endcsname
\else
  \errmessage{Undefined Kyedconstant #1.}%
\fi}
}

\newcommand{\velrotstokes}{v}
\newcommand{\velrotstokesrot}{v^\bot}
\newcommand{\presrotstokes}{q}

\newcommand{\velsstokes}{v}
\newcommand{\pressstokes}{q}
\newcommand{\veltp}{w}
\newcommand{\prestp}{\pi}

\newcommand{\rhssstokes}{f}
\newcommand{\rhssstokestensor}{F}
\newcommand{\rhsrot}{f}
\newcommand{\rhsrottensor}{F}

\newcommand{\rhstpgeneric}{h}
\newcommand{\rhstpgenerictensor}{H}
\newcommand{\rhstpproj}{\overline{h}}
\newcommand{\rhstpprojtensor}{\overline{H}}
\newcommand{\restfunc}{r}
\newcommand{\restfunctensor}{R}
\renewcommand{\ln}{\log}

\theoremstyle{plain}
\newtheorem{thm}{Theorem}[section]

\newtheorem{lem}[thm]{Lemma}

\theoremstyle{remark}

\begin{document}
%%%%%%%%%%%%%%%%%%%%%%%%%%%%%%%%%%%%%%%%%%%%%%%%%%%%%%%%%%%%%%
%%          Title, author, date, abstract, etc.             %%
%%%%%%%%%%%%%%%%%%%%%%%%%%%%%%%%%%%%%%%%%%%%%%%%%%%%%%%%%%%%%%
\title{On the asymptotic structure of steady Stokes and Navier-Stokes flows around a rotating two-dimensional body}

\author{
Toshiaki Hishida\thanks{
Supported in part by Grant-in-Aid for Scientific Research 18K03363
from JSPS}\\
Graduate School of Mathematics\\
Nagoya University\\
Nagoya 464-8602, Japan\\
Email: {\texttt hishida@math.nagoya-u.ac.jp}
\and
Mads Kyed\\ 
Fachbereich Mathematik\\
Technische Universit\"at Darmstadt\\
Schlossgartenstr. 7, 64289 Darmstadt, Germany\\
Email: {\texttt kyed@mathematik.tu-darmstadt.de}
}

\date{\today}
\maketitle

\begin{abstract}
We establish pointwise decay estimates for the velocity field of a steady two-dimensional Stokes flow around a rotating body via a new approach
rather than analysis adopted in the previous literature \cite{HMN}, \cite{Hi}.
The novelty is to analyze the singular behavior of the constants in these estimates with respect to the angular velocity of the body,
where such singularity is reasonable on account of the Stokes paradox.
We then employ the estimates to identify the asymptotic structure 
at infinity of a steady scale-critical Navier-Stokes flow, 
being assumed to be small, around a rotating body.
It is proved that the leading term is given by a self-similar Navier-Stokes
flow which exhibits a circular profile $x^\perp/|x|^2$ and whose
coefficient is the torque acting on the body.
\end{abstract}


%%%%%%%%%%%%%%%%%%%%%%%%%%%%%%%%%%%%%%%%%%%%%%%%%%%%%%%%%%%%%%%%%%%%%%%%%%%%%%%%%%%%%%%%%%%%%%%%%%%%%
%%          Global Constant Counters                                                               %%
%% Conmands:                                                                                       %%
%%    \newCCtr[C]{const}	defines a constant counter named ``const'' with the symbol ``C''   %% 
%%    \Cc{const}		inserts the symbol (with number) of constant counter ``const''     %%
%%    \Cc[ref]{const}		same as \Cc{const} but adds the reference ``ref''                  %%
%%    \const{ref}		inserts the numbered constant ``ref''                              %%
%%%%%%%%%%%%%%%%%%%%%%%%%%%%%%%%%%%%%%%%%%%%%%%%%%%%%%%%%%%%%%%%%%%%%%%%%%%%%%%%%%%%%%%%%%%%%%%%%%%%%
\newCCtr[C]{C}
\newCCtr[M]{M}
\newCCtr[\epsilon]{eps}
\CcSetCtr{eps}{-1}
\newCCtr[c]{c}
\let\oldproof\proof
\def\proof{\CcSetCtr{c}{-1}\oldproof} 

%%%%%%%%%%%%%%%%%%%%%%%%%%%%%%%%%%%%%%%%%%%%%%%%%%%%%%%%%%%%%%
%%          Main document                                   %%
%%%%%%%%%%%%%%%%%%%%%%%%%%%%%%%%%%%%%%%%%%%%%%%%%%%%%%%%%%%%%%
\section{Introduction and the main result}
\label{intro}

Consider the flow of an incompressible viscous fluid, governed by the Navier-Stokes equations, around a two-dimensional
rigid body, which occupies a simply connected bounded domain
$\body\subset\R^2$.
The fluid then occupies the exterior domain
$\Omega:=\R^2\setminus\overline{\body}$, whose boundary
$\partial\Omega=\partial\body$ we assume to be sufficiently smooth.
Analysis of the asymptotic behavior at spatial infinity of a steady Navier-Stokes flow 
in 2D is very challenging and substantially more difficult than the corresponding 3D problem.
One of the difficulties stems from the Stokes paradox,
which states that a 2D Stokes flow cannot be bounded near infinity
unless the net force vanishes (see Chang and Finn \cite{CF}). 
The Stokes paradox is rooted in the lack of decay of the 2D Stokes fundamental solution, which actually grows
logarithmically.
Therefore, the Stokes linearization is not well suited as a basis for investigation of
the nonlinear Navier-Stokes problem in this case.
Although one can find a solution in the Leray class (with finite Dirichlet integral) to the steady-state 2D Navier-Stokes equations if a prescribed boundary condition at infinity is disregarded,
see the celebrated paper by Leray \cite{L},
the lack of a suitable linearization means that very little is known about its asymptotic behavior at spatial infinity. Indeed, this question
remains one of the outstanding open problems in the field of mathematical fluid mechanics to date.

When the body $\body$ is translating with constant velocity, the steady
motion in a frame attached to the moving body is governed
by the Navier-Stokes equations with an Oseen term. The linearization hereof is an Oseen system rather than a Stokes system.
Since the Oseen fundamental solution has an anisotropic decay
structure (with wake), the Stokes paradox is not an issue in this case.
Moreover, the Oseen fundamental solution describes the leading profile at infinity
of a Navier-Stokes flow in the Leray class without any smallness
condition; see Galdi \cite[XII.8]{G-b}.

In this paper we consider a different motion of the body $\body$, namely rotation with a constant angular velocity.
If the body $\body$ is rotating with constant angular velocity
$\angvel\in\R\setminus\{0\}$, the motion of the flow is governed by the Navier-Stokes system
\[
\partial_tv+v\cdot\nabla_yv=\Delta_yv-\nabla_yq+g, \qquad
\Div_y v=0
\]
in $\Omega(t)=\{y=Q(t)x|\, x\in\Omega\}$, where
\begin{equation}
\rotmatrix(t)\coloneqq\left(
\begin{array}{cc}
\cos at & -\sin at \\
\sin at & \cos at
\end{array}
\right).
\label{rot-mat}
\end{equation}
Here, $v=(v_1(t,y),v_2(t,y))^\top$ and $q(t,y)$ denote the
unknown velocity field and pressure of the fluid, respectively,
while $g=(g_1(t,y),g_2(t,y))^\top$ is a given external force.
Throughout this paper, $(\cdot)^\top$ denotes the transpose of
vectors and matrices, and all vectors are column ones.
By a change of coordinates
\begin{equation}
u(t,x)\coloneqq Q(t)^\top v\bp{t,Q(t)x}, \quad
p(t,x)\coloneqq q\bp{t,Q(t)x},\quad
f(t,x)\coloneqq Q(t)^\top g\bp{t,Q(t)x},
\label{trans}
\end{equation}
we can rewrite the system in a frame attached to the body $\body$,
which then reads
\[
\partial_tu+u\cdot\nabla u
=\Delta u+a(x^\perp\cdot\nabla u-u^\perp)-\nabla p+f, \qquad
\Div u=0
\]
in the time-independent domain $\Omega$, where
$x^\perp \coloneqq (-x_2,x_1)^\top, u^\perp \coloneqq(-u_2,u_1)^\top$.
In this paper we assume that $f=f(x)$ is independent of $t$ and study
the steady problem
\begin{equation}
-\Delta u-a(x^\perp\cdot\nabla u-u^\perp)+\nabla p+u\cdot\nabla u=f, \qquad
\Div u=0
\label{NS}
\end{equation}
in $\Omega$.
Usually, the no-slip boundary condition
$u|_{\partial\Omega}=ax^\perp$ is imposed, but 
it is better to understand the asymptotic structure at spatial infinity
of solutions to \eqref{NS} only from the equation without specifying
the boundary condition at $\partial\Omega$.

%\begin{align}\label{intro_NavierStokesRotFrame}
%\begin{pdeq}
%&-\Delta\velrotext - \angvel(\nsnonlinb{\xrot}{\velrotext}-\velrotextrot)+ \nsnonlin{\velrotext} + \grad\presrotext = \rhsrotext && \tin\Omega,\\
%&\Div\velrotext = 0 && \tin \Omega,\\
%&\lim_{\snorm{x}\ra\infty} \velrotext(x)=0,
%\end{pdeq}
%\end{align}

In contrast to the case $a=0$, we have a chance to find a generic flow
that is at rest at infinity under an appropriate condition
on the external force $f$.
In fact, for the linearized system
\begin{equation}
-\Delta u-a(x^\perp\cdot\nabla u-u^\perp)+\nabla p=f, \qquad
\Div u=0,
\label{linear}
\end{equation}
it was discovered first by Hishida %\cite{Hishida2016}
\cite{Hi} that the oscillation due to rotation of the body
leads to the resolution of the Stokes
paradox on account of the decay structure of the 
fundamental solution associated with \eqref{linear}.
More precisely, if $\{u,p\}$ satisfies \eqref{linear} in $\Omega$
and $u(x)=o(|x|)$ at infinity (to exclude polynomials except
constant vectors), one can show that
$u(x)-u_\infty=O(|x|^{-1})$ as $|x|\to\infty$ for some constant
vector $u_\infty\in\R^2$, and even the asymptotic representation
\begin{equation}
u(x)-u_\infty=M\frac{x^\perp}{4\pi |x|^2}+\beta\frac{-x}{2\pi |x|^2}
+o(|x|^{-1})
\qquad\mbox{as $|x|\to\infty$}
\label{linear-str}
\end{equation}
with
\[
M \coloneqq\int_{\partial\Omega}y^\perp\cdot\{(T(u,p)+au\otimes y^\perp)\nu\}d\sigma
+\int_\Omega y^\perp\cdot f\,dy, \qquad
\beta \coloneqq\int_{\partial\Omega}\nu\cdot u\,d\sigma,
\]
provided
$f(x)=o\bp{|x|^{-3}(\log |x|)^{-1}}$ as $|x|\to\infty$,
where
$T(u,p)=\nabla u+(\nabla u)^\top-p I$
denotes the Cauchy stress tensor
($I\in\mathbb R^{2\times 2}$ being the identity matrix)
and $\nu$ the outward unit normal to $\partial\Omega$.
The second term $-\beta x/(2\pi |x|^2)$ in \eqref{linear-str} is nothing but the flux carrier, and
by subtracting this carrier we can reduce the problem to the one 
subject to
\begin{equation}
\int_{\partial\Omega}\nu\cdot u\,d\sigma=0.
\label{zero-flux}
\end{equation}
Observe that the no-slip condition 
$u|_{\partial\Omega}=ax^\perp$ mentioned above
fulfills \eqref{zero-flux}.
We may thus conclude that the essential profile is the circular flow
$x^\perp/(4\pi |x|^2)$ in \eqref{linear-str}, and that the rate of decay
is controlled by the torque $M$ (not by the force).
The proof in \cite{Hi}
relies on a detailed analysis of the fundamental solution
(of two variables $x,\,y$ since the elliptic operator in \eqref{linear}
has a variable coefficient)
whose leading term for $|x|>2|y|$ is
$x^\perp\otimes y^\perp/(4\pi |x|^2)$.

Indeed this is linear analysis, but it is remarkable that
the profile in the asymptotic representation \eqref{linear-str},
more specifically, the pair
\begin{equation}
U(x) \coloneqq\frac{cx^\perp}{|x|^2}, \qquad
P(x) \coloneqq\frac{-c^2}{2|x|^2} \qquad (c\in\mathbb R),
\label{circu}
\end{equation}
is itself a homogeneous Navier-Stokes flow in $\mathbb R^2\setminus\{0\}$
(of degree $(-1)$ for the velocity),
that is, a self-similar Navier-Stokes flow in 2D.
The pair $\set{U,P}$ also solves \eqref{NS} with $f=0$ in
$\mathbb R^2\setminus\{0\}$ since
$x^\perp\cdot\nabla U=U^\perp$.
Regardless of spatial dimension, for steady Navier-Stokes flows
that decay to zero with the scale invariant rate $O\bp{|x|^{-1}}$,
the nonlinearity is balanced with the linear part. For such flows,
one may expect that its leading term at infinity is described by
a certain self-similar flow, even if the magnitude of the flow is large.
This is indeed the case for small Navier-Stokes flows in 3D both
when the body is at rest \cite{KSv11, NaPi, KMT, Hi-hb}
and when the body is rotating with a constant angular velocity
\cite{FH11-2, FGK}. We further refer to the paper \cite{S} by
\v Sver\'ak, who gave an insight into self-similar Navier-Stokes flows.
In contrast, in the case of a body translating with constant velocity the Oseen fundamental
solution is the leading profile even for large Navier-Stokes flows; see \cite{G-b} as well as \cite{Fi73}
and the references therein.

For the 2D problem under consideration here, the linear analysis developed in \cite{Hi}
is not sufficient to analyze the Navier-Stokes system \eqref{NS}
because the estimate in \cite{Hi} of the remainder term in the asymptotic representation \eqref{linear-str} with respect to the angular velocity
is too singular like $O\bp{|a|^{-1}}$.
In a more recent paper, Higaki, Maekawa and Nakahara \cite{HMN} obtained 
a nice estimate of this remainder with less singular behavior
for $a\to 0$, and applied it to \eqref{NS}.
Roughly speaking, their theorem asserts that if $|a|$ is small and
the decaying force $f(x)$ of divergence form
is also small compared to some rate of $|a|$
(which is almost $|a|^{1/2}$),
problem \eqref{NS} in $\Omega$ subject to the no-slip condition
$u|_{\partial\Omega}=ax^\perp$ admits a unique solution $u(x)$
with leading profile
$x^\perp/|x|^2$ whose coefficient is the torque.
We also mention another existence theorem for \eqref{NS} 
with arbitrary $a\in\mathbb R\setminus\{0\}$ (together with
a boundary layer analysis for $|a|\to\infty$) due to
Gallagher, Higaki and Maekawa \cite{GHM} when the obstacle is exactly a disk.

The aim in the following is two-fold. Firstly, we introduce a new and simplified approach towards a linear
theory (a priori estimates in suitable function spaces) for \eqref{linear} that is optimal with respect to the singularity for
$a\ra 0$.
Secondly, we seek to employ these estimates to establish an asymptotic representation of given solution to 
\eqref{NS} that decays like $O\bp{\snorm{x}^{-1}}$. The latter is obtained under a smallness condition.

In the first part, we provide a different and considerably shorter proof
of the resolution of the Stokes paradox than the previous one in \cite{Hi}.
The strategy is to express a steady solution to \eqref{linear} in the coordinates of the inertial frame
using the transformation \eqref{rot-mat} (as was done first by Galdi \cite{G03}).
In the inertial frame of reference, the solution is time-periodic. After 
splitting this time-periodic solution into a steady part, which is given by the average over the period, and a 
purely periodic part, whose average over the period vanishes,
we obtain our result by analyzing each part separately.
This idea was adopted by Galdi \cite{G13} and
has been developed in terms of time-periodic fundamental
solutions introduced by Kyed \cite{Ky}.
Our procedure yields a very useful new estimate 
(rather than \cite[Theorem 3.1]{HMN})
for solutions to the linearized system \eqref{linear} in the whole plane $\mathbb R^2$,
see Theorem \ref{linear-thm},
when the torque of $f=f_0+\mbox{div $F$}$
with $F=(F_{ij})$ vanishes, that is,
\begin{equation*}
\int_{\mathbb R^2}y^\perp\cdot f_0\,\dy
+\int_{\mathbb R^2}(F_{12}-F_{21})\,\dy=0. 
%\label{zero-torque}
\end{equation*}
The estimate reveals that the leading term in an asymptotic expansion of the
velocity field comes only from the steady part,
while the singular behavior with respect to $a\ra 0$ arises only from
the purely periodic part.
Due to zero average of this latter part,
its several fine decay properties for $|x|\to\infty$ have been
established in \cite{Ky} and \cite{EKy} via pointwise estimates of the
time-periodic Stokes fundamental solution.
However, the estimates in \cite{Ky} and \cite{EKy} are not sufficient to adequately describe the singular behavior with respect to $a\ra 0$.
For this purpose, one also needs the singular behavior
of the time-periodic fundamental solution around the origin $x=0$,
which is not provided in \cite{Ky} or \cite{EKy}, and indeed difficult to obtain in the time-periodic case (in contrast to classical
fundamental solutions).
In Lemma \ref{FundsolComplPointwiseDecayLem}, we establish such an estimate, 
which even describes simultaneously the decay at large
distance and around the origin.
Estimates of the purely periodic part with faster decay
rate involve more singular behavior for $a\to 0$ as the price.
Using Lemma \ref{FundsolComplPointwiseDecayLem} and a scaling argument,
we are able to quantify this trade-off, to be precise,
given $\delta\in (0,1)$, we find a reasonable
singular behavior to get the decay of the purely periodic part like
$O(|x|^{-(1+\delta)})$ uniformly in $t$,
see Lemma \ref{ConvolutionWithPurelyPeriodicFundSolLemma}.

In the second part of this paper, 
we consider arbitrary solutions to \eqref{NS} in $\Omega$ that decay
with the scale invariant rate $O\bp{|x|^{-1}}$
without specifying any boundary condition except \eqref{zero-flux}.
It is interesting to ask whether they exhibit
the same asymptotic structure as the solution constructed in \cite{HMN}
no matter how they are constructed.
As the main theorem of the paper, and as a nice application of the linear theory developed in the first part, we give an
affirmative answer (however, in the small) to this question.
\begin{thm}
Let $\Omega\subset\R^2$ be an exterior domain with $C^2$-boundary,
and let $a\in\mathbb R\setminus\{0\}$.
Given $\delta\in (0,1/2)$ and $R>e$ satisfying
$\mathbb R^2\setminus\Omega\subset B_R(0):=\{x\in\mathbb R^2|\,|x|<R\}$,
there are positive constants
$\kappa=\kappa(\delta)$ (independent of $R$ and $a$) and
$\mu=\mu(\delta,R)$ (independent of $a$) such that the following holds:
For every solution (smooth solution for simplicity)
$\{u,p\}\in H^1_{loc}(\overline\Omega)\times L^2_{loc}(\overline\Omega)$
to \eqref{NS} with
$f\in L^2_{loc}(\overline\Omega)$ subject to \eqref{zero-flux}
which satisfies
\begin{align}
\begin{pdeq}
&\displaystyle{\bp{1+|a|^{-\delta/2}}\sup_{|x|\geq R}|x||u(x)|\leq\kappa,} \\
&\displaystyle{(1+|a|^{-(\delta+1/2)})\sup_{|x|\geq R}|x|^{3+\delta}
|f(x)|\leq\kappa,} \\
&\bp{|a|+|a|^{-(\delta+1/2)}}\,|M|\leq\mu,\\
&\bp{|a|+|a|^{-(\delta+1/2)}}\,{\sup_{R<|x|<2R} \bp{\snorm{u(x)} + \snorm{\nabla u(x)} + \snorm{\nabla^2 u(x)} + \snorm{p(x)}}}\leq\mu,\\
\end{pdeq}
\label{small}
\end{align}
where 
\begin{equation}
M\coloneqq\int_{\partial\Omega}y^\perp\cdot\bcp{\bp{T(u,p)+au\otimes y^\perp-u\otimes u}\nu}d\sigma+\int_\Omega y^\perp\cdot f\,dy
\label{torque}
\end{equation}
(the total torque),
we have the asymptotic representation
\begin{equation}
u(x)=M\frac{x^\perp}{4\pi|x|^2}+O\bp{|x|^{-(1+\delta)}}\qquad
\mbox{as $|x|\to\infty$}.
\label{structure}
\end{equation}
\label{main}
\end{thm}

Note that the boundary integral in \eqref{torque} is understood as
$\langle y^\perp,(\cdots)\nu\rangle_{\partial\Omega}$ since
$(\cdots)\nu\in H^{-1/2}(\partial\Omega):=H^{1/2}(\partial\Omega)^*$
by the normal trace theorem
on account of the assumptions on the regularity of
$\{u,p\}$ and $f$ up to $\partial\Omega$.

The influence of $a$ in the smallness condition \eqref{small} is a delicate matter if the solution itself depends on $a$.
This is indeed the situation with the most natural boundary condition
$u|_{\partial\Omega}=ax^\perp$, \textit{i.e.}, the no-slip condition. In this case, the terms in \eqref{small} depending on $u$, that is,
\begin{align*}
\sup_{|x|\geq R}|x||u(x)|,\quad M,\quad {\sup_{R<|x|<2R} \bp{\snorm{u(x)} + \snorm{\nabla u(x)} + \snorm{\nabla^2 u(x)} + \snorm{p(x)}}},
\end{align*}
are controlled
by $\snorm{a}$ and a magnitude of $f$. Importantly, since $\delta+1/2<1$, the smallness condition
\eqref{small} is satisfied in this case when the data $a$ and $f$ are sufficiently small.

In the next section
we study the Stokes system in steady and time-periodic regimes, separately.
Combining those studies in both regimes, in Section \ref{linear-problem},
we develop the linear theory for \eqref{linear} in the whole plane
$\mathbb R^2$.
The final section is devoted to the proof of Theorem \ref{main}.

\section{Stokes system}
\label{stokes-system}
%\section{Preliminaries}

We make use of the Einstein summation convention and
implicitly sum over all repeated indices. Moreover, we abbreviate $\partial_j:=\partial_{x_j}$.
Given $\alpha\in(0,\infty)$, we define the Banach spaces
\begin{align*}
&\wspace{\alpha}(\R^2) := \setcl{f\in\LR{\infty}(\R^2)}{\norm{f}_{\wspace{\alpha}}<\infty}, \\
&\wspacelog{\alpha}(\R^2) := \setcl{f\in\LR{\infty}(\R^2)}{\norm{f}_{\wspacelog{\alpha}}<\infty},
\end{align*}
endowed with
\begin{equation*}
\begin{split}
&\norm{f}_{\wspace{\alpha}} 
:= \sup_{x\in\R^2} (1+\snorm{x})^\alpha\snorm{f(x)},  \\
&\norm{f}_{\wspacelog{\alpha}}
:= \sup_{x\in\R^2}
(1+\snorm{x})^\alpha (\log (e+\snorm{x}))\snorm{f(x)},
\end{split}
\end{equation*}
respectively.

\subsection{Steady-State Stokes system}

Consider the steady-state Stokes system
\begin{align}\label{SteadyStokes}
\begin{pdeq}
&-\Delta\velsstokes + \grad\pressstokes = \rhssstokes && \tin\R^2, \\
&\Div\velsstokes =0 && \tin\R^2
\end{pdeq}
\end{align}
and recall the fundamental solution $\fundsolstokesvel\in\TDR(\R^2)^{2\times 2}$ to \eqref{SteadyStokes} given by the function 
\begin{align}\label{SteadyStateStokesFundEq}
&\fundsolstokesvel_{ij}(x)\coloneqq 
\frac{1}{4\pi}\Bp{\delta_{ij}\log\bp{\snorm{x}^{-1}}+\frac{x_ix_j}{\snorm{x}^2} }. 
\end{align}
We need the following expansion of convolutions with $\fundsolstokesvel$:

\begin{lem}\label{SteadyStokesRepresentationFundsol}
Let $\delta\in(0,1)$, $\rhssstokes\in\wspacelog{3+\delta}(\R^2)^2$ and $\rhssstokestensor\in\wspace{2+\delta}(\R^2)^{2\times2}$. Then ($i=1,2$) 
\begin{align}\label{SteadyStokesRepresentationFundsol_Rep}
\begin{aligned}
&\fundsolstokesvel_{il}*\rhssstokes_l(x) = 
\Bb{\int_{\R^2}\rhssstokes(y)\,\dy}_l\fundsolstokesvel_{il}(x)
- \Bb{\int_{\R^2}\rhssstokes(y)\otimes y\,\dy}_{lj}\partial_j\fundsolstokesvel_{il}(x) 
+ \restfunc_i(x),\\
&\sup_{|x|\geq e}|x|^{1+\delta}|\restfunc_i(x)|
%\norm{\restfunc_i}_{\wspace{1+\delta}}
\leq \Cc[SteadyStokesRepresentationFundsol_RestfuncConst]{C} \norm{\rhssstokes}_{\wspacelog{3+\delta}}
\end{aligned}
\end{align}
and 
\begin{align}\label{SteadyStokesRepresentationFundsol_TensorRep}
\begin{aligned}
&\partial_j\fundsolstokesvel_{il}*\rhssstokestensor_{lj}(x) = 
\Bb{\int_{\R^2}\rhssstokestensor(y)\,\dy}_{lj}\partial_j\fundsolstokesvel_{il}(x)+\restfunctensor_i(x),\\
&\sup_{|x|\geq e}|x|^{1+\delta}|\restfunctensor_i(x)|
%\norm{\restfunctensor_i}_{\wspace{1+\delta}} 
\leq \Cc[SteadyStokesRepresentationFundsol_RestfunctensorConst]{C} \norm{\rhssstokestensor}_{\wspace{2+\delta}}.
\end{aligned}
\end{align}
\end{lem}
   
\begin{proof}
Let $|x|\geq e$.
We fix $i\in \{1,2\}$ and decompose $(\Gamma^S*f)_i$ as
\begin{equation*}
\begin{split}
(\Gamma^S*f)_i(x)
&=\left(\int_{|y|<|x|/2}+\int_{|x|/2\leq |y|\leq 2|x|}+\int_{|y|>2|x|}\right)
\Gamma^S_{il}(x-y)f_l(y)\,\dy  \\
&=:I_1+I_2+I_3.
\end{split}
\end{equation*}
We show that the leading and second order terms with respect to an asymptotic expansion $\snorm{x}\ra\infty$
come from $I_1$.
To this end, we decompose $I_1$ as
\begin{equation}
I_1=\Gamma^S_{il}(x)\int_{|y|<|x|/2}f_l(y)\,\dy
+(\partial_j\Gamma_{il})(x)\int_{|y|<|x|/2}(-y_j)f_l(y)\,\dy 
+\widetilde r_i(x),
\label{I-1}
\end{equation}
where
\[
\widetilde r_i(x) 
%&=\sum_{j,k=1}^2\int_{|z|<|y|/2}(-z_k)h_j(z) \int_0^1  
%\left\{\left(\partial_kE_{ij}\right)(y-sz)-\left(\partial_kE_{ij}\right)(y)
%\right\}\,ds\,dz  \\
%&=\sum_{j,k,l=1}^2\int_{|z|<|y|/2} z_lz_k h_j(z)  
%\int_0^1\int_0^s\left(\partial_l\partial_kE_{ij}\right)(y-\tau z)\,d\tau\,ds
%\,dz \\
:=\int_{|y|<|x|/2} y_ky_j f_l(y)
\int_0^1 (1-\tau)\, \bp{\partial_k\partial_j\Gamma^S_{il}} (x-\tau y)\,
d\tau\,dy.
\]
From \eqref{SteadyStateStokesFundEq} we directly obtain 
\[
|\left(\partial_k\partial_j\Gamma^S_{il}\right)(x-\tau y)|
\leq\frac{C}{|x-\tau y|^2}\leq\frac{C}{|x|^2} \qquad (|x|>2|y|),
\]
which yields
\begin{equation}
|\widetilde r_i(x)|
\leq\frac{C}{|x|^2}\int_{|y|<|x|/2}|y|^2|f(y)|\,\dy
%\leq \frac{C}{|x|^2}\int_0^{|x|/2}\rho^3(1+\rho^{3+\delta})^{-1}\,d\rho
%\|f\|_{\wspacelog{3+\delta}}
\leq C|x|^{-(1+\delta)}\|f\|_{\wspace{3+\delta}}.
\label{tilde-r}
\end{equation}
Since
\begin{equation*} 
\begin{split}
&\left|\Gamma^S_{il}(x)\int_{|y|\geq |x|/2}f_l(y)\,\dy\right| \\
&\qquad\leq C(1+\log |x|)\|f\|_{\wspacelog{3+\delta}}
\int_{|y|\geq |x|/2}
(1+|y|)^{-(3+\delta)}(\log(e+|y|)^{-1}\,\dy  \\
%&\leq \frac{C(1+\log |y|)}{\log\frac{|y|}{2}}\int_{|y|/2}^\infty\frac{dr}{r^2}
%\sup_{|z|\geq |y|/2}|z|^3(\log |z|)|h(z)|  \\ 
&\qquad\leq C(1+|x|)^{-(1+\delta)}\|f\|_{\wspacelog{3+\delta}}
\end{split} 
\end{equation*}
and since
\[
\left|(\partial_j\Gamma^S_{il})(x)\int_{|y|\geq |x|/2}(-y_j)f_l(y)\,\dy
\right| \\
\leq C|x|^{-1}(1+|x|)^{-\delta}\|f\|_{\wspace{3+\delta}}
\]
it follows from \eqref{I-1} and \eqref{tilde-r} that
\begin{equation}
I_1=\alpha_l(f)\Gamma^S_{il}(x) 
+\beta_{lj}(f) \left(\partial_j\Gamma^S_{il}\right)(x) +\overline{r}_i(x)
\label{asym-I-1} 
\end{equation}
with
\[
\alpha_l(f):=\int_{\mathbb R^2}f_l(y)\,\dy, \qquad
\beta_{lj}(f):=\int_{\mathbb R^2}(-y_j)f_l(y)\,\dy, \qquad
|\overline{r}_i(x)|\leq C|x|^{-(1+\delta)}\|f\|_{\wspacelog{3+\delta}}.
\]
We turn to the estimate of $I_2$.
Using
\[
\int_{|y|\leq 2|x|}|\Gamma^S(x-y)|\,\dy
\leq C\int_{|y-x|\leq 3|x|}\left(1+\left|\log\frac{1}{|y-x|}\right|\right)\dy
\leq C|x|^2(1+\log |x|),
\]
we obtain
\begin{equation}
\begin{split}
|I_2|
&\leq C(1+|x|)^{-(3+\delta)}  \Bp{\log\frac{|x|}{2}}^{-1}\,
\|f\|_{\wspacelog{3+\delta}}
\int_{|y|\leq 2|x|}|\Gamma^S(x-y)|\,\dy  \\
&\leq C(1+|x|)^{-(1+\delta)}\|f\|_{\wspacelog{3+\delta}}.
\end{split} 
\label{I-2} 
\end{equation}
Finally, we have
\begin{equation}  
%\begin{split}
|I_3| 
\leq C\int_{|y|> 2|x|}(1+\log |y|)|f(y)|\,\dy
%\leq C\|f\|_{\wspacelog{3+\delta}}\int_{|y|>2|x|}|y|^{-(3+\delta)}\,\dy
\leq C(1+|x|)^{-(1+\delta)}\|f\|_{\wspacelog{3+\delta}}.
%\end{split} 
\label{I-3}
\end{equation}
We collect \eqref{asym-I-1}, \eqref{I-2} and \eqref{I-3} to conclude
\eqref{SteadyStokesRepresentationFundsol_Rep}.
The other representation
\eqref{SteadyStokesRepresentationFundsol_TensorRep} is proved
in a similar way.
\end{proof}

\subsection{Time-periodic Stokes system}

By $\torus_\per\coloneqq\R/\per\Z$ we denote torus groups
for ${\cal T}>0$.  
We consider $\grpper$ as a locally compact abelian group with 
a topology and differentiable structure inherited from $\R\times\R^2$ via the quotient mapping
$\quotientmap:\R\times\R^2\ra\grpper$, $\quotientmap(t,x)\coloneqq\bp{[t],x}$. We let $\dt$ denote the normalized Haar measure on $\torusper$, which means that
\begin{align*}
\int_{\torusper} f(t)\,\dt = \frac{1}{\per}\int_0^\per f(t)\,\dt
\end{align*}
when $\torusper$ is identified with the interval $[0,\per)$ in the canonical way.

Taking $\torus_\per$ as a time axis, we can conveniently formulate
the $\per$-time-periodic Stokes problem in the two-dimensional whole-space as:
\begin{align}\label{TimePeriodicStokesTSetting}
\begin{pdeq}
&\partial_t\veltp-\Delta\veltp + \grad\prestp = \rhstpgeneric && \tin\grp, \\
&\Div\veltp =0 && \tin\grp.
\end{pdeq}
\end{align}

We shall investigate \eqref{TimePeriodicStokesTSetting} using the Fourier transform $\FT_\grpper$ in the framework of the Schwartz-Bruhat space $\SR(\grpper)$ and corresponding space of tempered distributions
$\TDR(\grpper)$; see for example 
\cite{Ky, EKy}
%\cite{Kyed_FundsolTPStokes2016,EiterKyed_etpfslns} 
for more details.
We identify the dual group of ${\grpper}$ with $\Z\times\R^2$ and denote points in the dual group by $(k,\xi)\in\Z\times\R^2$. 
The Fourier transform $\FT_{\grpper}:\SR(\grpper)\ra\SR(\dualgrp)$ and its inverse are then given by
\begin{align*}
&\FT_\grpper\nb{u}(k,\xi)\coloneqq\int_\torusper\int_{\R^2} u(t,x)\,\e^{-ix\cdot\xi-ik\perf t}\,\dx\dt,\\
&\iFT_\grpper\nb{w}(t,x)\coloneqq\sum_{k\in\Z}\,\int_{\R^2} w(k,\xi)\,\e^{ix\cdot\xi+ik\perf t}\,\dxi,
\end{align*}
respectively, provided the Lebesgue measure $\dxi$ is normalized appropriately. By duality, $\FT_\grpper$ extends to a homeomorphism $\FT_\grpper:\TDR(\grpper)\ra\TDR(\dualgrp)$. Observe that $\FT_\grpper=\FT_{\torus_\per}\circ\FT_{\R^2}$.

The concept of a  fundamental solution to the time-periodic Stokes equations was introduced in \cite{Ky}
%\cite{Kyed_FundsolTPStokes2016}
as a distribution $(\fundsoltpvelper,\fundsoltppresper)\in\TDR(\grpper)^{2\times 2}\times\TDR(\grpper)^{2}$ satisfying
\begin{align}\label{tpFundsolEq}
\begin{pdeq}
&\partial_t\fundsoltpvelper_{ij}-\Delta\fundsoltpvelper_{ij} + \partial_i \fundsoltppresper_j = \delta_{ij}\,\delta_{\grpper},  \\
&\partial_j\fundsoltpvelper_{ij} =0.
\end{pdeq}
\end{align} 
Here, $\delta_{ij}$ and $\delta_{\grpper}$ denote the Kronecker delta and delta distribution, respectively.
We can identify a time-periodic fundamental solution as the sum of a fundamental solution to the steady-state Stokes problem and a remainder part we 
shall refer to as \emph{purely periodic} part. 
Employing the Fourier transform $\FT_\grpper$ in \eqref{tpFundsolEq}, we find as in \cite{Ky, EKy}
%\cite{Kyed_FundsolTPStokes2016,EiterKyed_etpfslns} 
a time-periodic fundamental solution given by 
\begin{align}\label{TimePeriodicFundsolDef}
\fundsoltpvelper \coloneqq \fundsolstokesvel\otimes \onedist_{\torus_\per} + \fundsolcomplper
\end{align}
with
\begin{align}\label{PurelyPeriodicTimePeriodicFundsolDef}
\fundsolcomplper \coloneqq \iFT_\grpper\Bb{ \frac{1-\delta_\Z(k)}{\snorm{\xi}^2 + i\perf k }\,\Bp{\idmatrix - \frac{\xi\otimes\xi}{\snorm{\xi}^2}}}\in\TDR(\grpper)^{2\times 2}.
\end{align}
Here, $\onedist_{\torus_\per}\in\TDR(\torus_\per)$ denotes the constant $1$, $\idmatrix\in\R^{2\times 2}$ the identity matrix, and $\delta_\Z$ the delta distribution on $\Z$ (which is simply
the function with $\delta_\Z(k)=1$ if $k=0$ and $\delta_\Z(k)=0$ if $k\neq0$). 
Given $\rhstpgeneric\in\SR(\grp)^2$,
a solution to the time-periodic Stokes problem \eqref{TimePeriodicStokesTSetting}
%$\rhstpgeneric\in\SR(\grp)^2$ 
is then given by 
$\veltp \coloneqq \fundsoltpvelper \convolutionpergroup \rhstpgeneric$, 
with component-wise convolution $\convolutionpergroup$ over the group $\grp$. From \eqref{TimePeriodicFundsolDef} we see that
\begin{align}\label{intro_ConvolutionWithTPFundsol2}
\veltp \coloneqq \fundsoltpvelper \convolutionpergroup  \rhstpgeneric = \fundsolstokesvel*_{\R^2}\Bp{\int_{\torusper} \rhstpgeneric(s,\cdot)\,\ds} + \fundsolcomplper\convolutionpergroup \rhstpgeneric.
\end{align}
%At the outset, a solution $\veltp$ given by the formula above belongs to $\TDR(\grp)^2$. To derive better regularity of $\veltp$, 

Let $\delta\in (0,1)$, then the issue of 
Lemma \ref{ConvolutionWithPurelyPeriodicFundSolLemma} below
is to quantify the dependence of decay estimates of the purely periodic
part like $O(|x|^{-(1+\delta)})$ on the period ${\cal T}$.
To this end, it is important to establish the following lemma,
which gives us pointwise estimates of
the purely periodic part $\fundsolcomplone$ of the fundamental solution
near $x=0$ simultaneously with those for large $|x|$.
\begin{lem}\label{FundsolComplPointwiseDecayLem}
For every $\gamma\in (0,1)$ and $p\in (1,\frac{1}{1-\gamma})$
there are constants
$C_3,\, C_4>0$ depending only on $\gamma,\, p$ such that
%Let $p\in(1,\infty)$. For all $\epsilon>0$
%
\begin{align}
&\norm{\fundsolcomplone(\cdot,x)}_{\LR{p}\np{\torusone}} \leq \Cc[FundsolComplPointwiseDecayLem_EstConst]{C}\,
\snorm{x}^{-2\gamma},
%\snorm{x}^{-2+\frac{2(1-\epsilon)}{p}},
\label{FundsolComplPointwiseDecayLem_PointwisedecayEst} \\
&\norm{\grad\fundsolcomplone(\cdot,x)}_{\LR{p}\np{\torusone}} \leq \Cc[FundsolComplPointwiseDecayLem_EstConst]{C}\, 
\snorm{x}^{-(1+2\gamma)},
%\snorm{x}^{-3+\frac{2(1-\epsilon)}{p}}
\label{FundsolComplPointwiseDecayLem_PointwisedecayEstGrad}
\end{align}
uniformly in $\mathbb R^2\setminus\{0\}$.
\end{lem}
\begin{proof}
We focus on \eqref{FundsolComplPointwiseDecayLem_PointwisedecayEstGrad}. We derive directly from \eqref{PurelyPeriodicTimePeriodicFundsolDef} the 
identity
\begin{align*}
\pdax\fundsolcomplone(\cdot,x) &= \iFT_\torusone\bb{ \torusmultiplier(k)\, \FT_\torusone\nb{\hgamma} }, 
\end{align*}
where
\begin{align*}
&\torusmultiplier:\Z\ra\C,\quad \torusmultiplier(k):= \bp{1-\delta_\Z(k)} \snorm{k}^\gamma \pdax \fundsolssrk(x),\\
&\fundsolssrk:=\iFT_{\R^2}\Bb{ \frac{1}{\snorm{\xi}^2 + i\perone k }\,\Bp{\idmatrix - \frac{\xi\otimes\xi}{\snorm{\xi}^2}}}\in\TDR(\R^2)^{2\times 2}\qquad(k\neq 0),
\end{align*}
and for $\gamma\in\np{0,1}$
\begin{align*}
\hgamma\in\TDR(\torusone),\quad \hgamma:=\iFT_{\torusone} \Bb{ \bp{1-\delta_\Z(k)} \snorm{k}^{-\gamma}}. 
\end{align*}
We shall establish \eqref{FundsolComplPointwiseDecayLem_PointwisedecayEstGrad} by showing that $\torusmultiplier$ is an $\LR{p}(\torusone)$ multiplier.
For this purpose, we utilize de Leeuw's Transference Principle in combination with Marcinkiewicz's Multiplier Theorem. Let
$\cutoff\in\CRi(\R)$ be a cut-off function with $\cutoff(\eta)=0$ for $\snorm{\eta}<\half$, and $\cutoff(\eta)=1$ for $\snorm{\eta}\geq 1$.
Put
\begin{align*}
\reallinemultiplier:\R\ra\C,\quad \reallinemultiplier(\eta):= \cutoff(\eta)\, \snorm{\eta}^\gamma\, \pdax \fundsolssreta(x).
\end{align*}
We compute as in \cite{BV}
%\cite{BorchersVarnhorn1993} 
to obtain 
\begin{align}\label{ResolventStokesFundSolId}
\begin{aligned}
&\fundsolssreta(x) = \frac{1}{2\pi}\Bb{\besseleone\bp{\sqrt{i\perone \eta}\,\snorm{x}}\idmatrix + \besseletwo\bp{\sqrt{i\perone \eta}\,\snorm{x}}\frac{x\otimes x}{\snorm{x}^2}},\\
&\besseleone(z) := \besselK_0(z) + z^{-1}\besselK_1(z) - z^{-2},\\
&\besseletwo(z) := -\besselK_0(z) - 2(z^{-1}\besselK_1(z) - z^{-2}),\\
\end{aligned}
\end{align}
where $\besselK_n$ denotes the modified Bessel function of order $n$, and $\sqrt{z}$ denotes the square root of $z$ with \emph{nonnegative} real part.
An expansion of $\besselK_1$ (see for example 
\cite[9.6.10--11]{AS})
%\cite[9.6.10--11]{AbramowitzStegunHandbook}), 
yields
$z^{-1}\besselK_1(z) - z^{-2} = \ln\bp{\half z}\entireP(z) + \entireQ(z)$
for two entire functions $\entireP$ and $\entireQ$. Recalling that $\besselK_0=\bigo\bp{\ln(z)}$ as $z\ra 0$ and $\besselK_0'(z)=-\besselK_1(z)=\bigo(z^{-1})$ (see for example 
\cite[9.6.8--9,9.6.27]{AS})
%\cite[9.6.8--9,9.6.27]{AbramowitzStegunHandbook}), 
we deduce ($j=1,2$)
\begin{align*}
\besselej(z)=\bigo\bp{\ln(z)},\quad
\besselej'(z)=\bigo\bp{\snorm{z}^{-1}},\quad 
\besselej''(z)=\bigo\bp{\snorm{z}^{-2}}\quad \tas\ z\ra 0. 
\end{align*}
This asymptotic behavior implies the following estimate for $\sqrt{\perone \snorm{\eta}}\,\snorm{x}\leq 1$:
\begin{align*}
\snorm{\reallinemultiplier(\eta)} 
&\leq \Cc{c} \snorm{\eta}^\gamma \bp{\snorm{x}^{-1} + \snorml{\ln(\sqrt{\perone \snorm{\eta}}\,\snorm{x})}\, \snorm{x}^{-1} }\\
&\leq \Cc{c} \snorm{\eta}^\gamma 
\{{\sqrt{\perone \snorm{\eta}}\,\snorm{x}\}^{-2\gamma}\,\snorm{x}^{-1} }
\leq \Cc{c} \snorm{x}^{-(1+2\gamma)}
\end{align*}
where $\Cclast{c}=\Cclast{c}(\gamma)$ is independent on $\eta$ and $x$.
Due to the exponential decay of modified Bessel functions as $z\ra\infty$ with $\realpart(z)>0$ (see for example 
\cite[9.2.3,9.6.4]{AS})
%\cite[9.2.3,9.6.4]{AbramowitzStegunHandbook}),
we further observe that
\begin{align*}
\besselej(z)=\bigo\bp{\snorm{z}^{-2}},\quad
\besselej'(z)=\bigo\bp{\snorm{z}^{-3}},\quad 
\besselej''(z)=\bigo\bp{\snorm{z}^{-4}}\quad \tas\ z\ra \infty \text{ with }\realpart(z)>0. 
\end{align*}  
We can thus estimate for $\sqrt{\perone \snorm{\eta}}\,\snorm{x}>1$:
\begin{align*}
\snorm{\reallinemultiplier(\eta)} 
&\leq \Cc{c} \snorm{\eta}^{\gamma-1} \snorm{x}^{-3} \\
&\leq \Cclast{c} \snorm{\eta}^{\gamma-1} \snorm{x}^{-3} \bp{\sqrt{\perone \snorm{\eta}}\,\snorm{x}}^{2(1-\gamma)}
\leq \Cc{c} \snorm{x}^{-(1+2\gamma)}
\end{align*}
where $\Cclast{c}=\Cclast{c}(\gamma)$ is independent on $\eta$ and $x$. The function 
$\eta\mapsto{\eta}\,\partial_\eta\reallinemultiplier(\eta)$ is estimated in a completely similar manner, and we conclude
\begin{align*}
\forall\eta\in\R:\quad  \snorm{\reallinemultiplier(\eta)} + \snorm{\eta\,\partial_\eta\reallinemultiplier(\eta)} 
&\leq \Cc[multipliernorm]{c} \snorm{x}^{-(1+2\gamma)}.
\end{align*}
By the Marcinkiewicz Multiplier Theorem (see for example 
\cite[Corollary 5.2.5]{Gr})
%\cite[Corollary 5.2.5]{Grafakos1}), 
$\reallinemultiplier$ is an $\LR{p}(\R)$ multiplier with operator norm bounded by $\const{multipliernorm} \snorm{x}^{-(1+2\gamma)}$.
Since $\torusmultiplier(k)=\reallinemultiplier(k)$ for all $k\in\Z$, the Transference Principle (see for example 
\cite[Theorem 3.6.7]{Gr})
%\cite[Corollary 3.6.7]{Grafakos1})
implies that $\torusmultiplier(k)$ is an $\LR{p}(\torusone)$ multiplier with its operator norm satisfying the same bound. We thus conclude
\begin{align*}
\norm{\pdax\fundsolcomplone(\cdot,x)}_{\LR{p}\np{\torusone}} &= \norm{\iFT_\torusone\bb{ \torusmultiplier(k)\, \FT_\torusone\nb{\hgamma} }}_{\LR{p}\np{\torusone}} \leq \const{multipliernorm} \snorm{x}^{-(1+2\gamma)}\norm{\hgamma}_p.
\end{align*}
It remains to show that $\norm{\hgamma}_p$ is finite
for $p\in (1,\frac{1}{1-\gamma})$.
To this end, we identify $\torusone$
 with the interval $(-\half,\half]$ and compute
\begin{align*}
\hgamma(t):=\iFT_{\torusone} \Bb{ \bp{1-\delta_\Z(k)} \snorm{k}^{-\gamma}} (t) = \Cc{c}\snorm{t}^{\gamma-1} + \ggamma(t),\quad t\in(-\half,\half],
\end{align*}
for some function $\ggamma\in \CR{\infty}(\torusone)$; see for example 
\cite[Example 3.1.19]{Gr}.
%\cite[Example 3.1.19]{Grafakos1}.
%We finally choose $\gamma:=1-\frac{1-\epsilon}{p}$ to conclude 
We thus conclude
\eqref{FundsolComplPointwiseDecayLem_PointwisedecayEstGrad}.
A completely similar argument yields \eqref{FundsolComplPointwiseDecayLem_PointwisedecayEst}.
\end{proof}
\begin{lem}\label{ConvolutionWithPurelyPeriodicFundSolLemma}
Let $\delta\in(0,1)$,
$\rhstpgeneric\in\LR{\infty}\bp{\torusper;\wspace{2+\delta}(\R^2)}^2$ and
$\rhstpgenerictensor\in\LR{\infty}\bp{\torusper;\wspace{2+\delta}(\R^2)}^{2\times2}$. Then ($i=1,2$)
\begin{align}
&\norm{\fundsolcomplper_{il}\convolutionpergroup \rhstpgeneric_{l}}_{\LR{\infty}\np{\torusper;\wspace{1+\delta}}} \leq 
\Cc[ConvolutionWithPurelyPeriodicFundSolLemma_Const]{C}\,\per^{\frac{1+\delta}{2}}\norm{\rhstpgeneric}_{\LR{\infty}\np{\torusper;\wspace{2+\delta}}},\label{ConvolutionWithPurelyPeriodicFundSolLemma_Estimate}\\
&\norm{\partial_j\fundsolcomplper_{il}\convolutionpergroup \rhstpgenerictensor_{lj}}_{\LR{\infty}\np{\torusper;\wspace{1+\delta}}} \leq 
\Cc[ConvolutionWithPurelyPeriodicFundSolLemma_GradConst]{C}\,\per^{\frac{\delta}{2}}\norm{\rhstpgenerictensor}_{\LR{\infty}\np{\torusper;\wspace{2+\delta}}},\label{ConvolutionWithPurelyPeriodicFundSolLemma_GradEstimate}
\end{align} 
where the convolution $\convolutionpergroup$ is taken with respect to time and space $(t,x)\in\torusper\times\R^2$, and the constants $\const{ConvolutionWithPurelyPeriodicFundSolLemma_Const}=\const{ConvolutionWithPurelyPeriodicFundSolLemma_Const}(\delta)$ and 
$\const{ConvolutionWithPurelyPeriodicFundSolLemma_GradConst}=\const{ConvolutionWithPurelyPeriodicFundSolLemma_GradConst}(\delta)$ 
are independent of $\per$.
\end{lem}

\begin{proof}
We focus on \eqref{ConvolutionWithPurelyPeriodicFundSolLemma_GradEstimate}.
The decay estimate in Lemma \ref{FundsolComplPointwiseDecayLem} ensures sufficient integrability of $\partial_j\fundsolcomplper$
for the convolution $\partial_j\fundsolcomplper_{il}\convolutionpergroup \rhstpgenerictensor_{jl}$ to be expressed in terms of a classical convolution integral 
\begin{align*}
\partial_j\fundsolcomplper_{il}*\rhstpgenerictensor_{lj}\,(t,x) 
= \int_{\torusper}\int_{\R^2} 
\partial_j\fundsolcomplper_{il}(t-s,x-y)\,\rhstpgenerictensor_{lj}(s,y)\,\dy\ds.
\end{align*}
One may verify directly from definition \eqref{PurelyPeriodicTimePeriodicFundsolDef} of $\fundsolcomplper$ the scaling property
\begin{equation}
\fundsolcomplper(t,x) = \fundsolcomplone(\per^{-1}t,\per^{-\half}x).
\label{funda-scaling}
\end{equation}
Inserting the above into the convolution integral, we obtain after a change of variables
\begin{align*}
\partial_j\fundsolcomplper_{il}*\rhstpgenerictensor_{lj}\,(t,x) 
&= \per^{-\half} \int_{\torusone}\int_{\R^2}  
\partial_j\fundsolcomplone_{il}(\per^{-1}t-s,\per^{-\half}x-\per^{-\half}y)\,\rhstpgenerictensor_{lj}(\per s, y)\,\dy\ds.
\end{align*}
We can thus employ H\"older's inequality and \eqref{FundsolComplPointwiseDecayLem_PointwisedecayEstGrad} for any $\gamma\in(0,1)$ and $p\in\bp{1,\frac{1}{1-\gamma}}$
%with $\epsilon=1-\frac{1}{p}$
to estimate
\begin{align*}
&\snorm{\partial_j\fundsolcomplper_{il}*\rhstpgenerictensor_{lj}\,(t,x)}\\
&\qquad\leq  
\per^{-\half} \int_{\R^2} \norm{\partial_j\fundsolcomplone_{il}(\cdot,\per^{-\half}x-\per^{-\half}y)}_{\LR{p}\np{\torusone}}\,
\norm{\rhstpgenerictensor_{lj}}_{\LR{\infty}\np{\torusper;\wspace{2+\delta}}}
\,\bp{1+\snorm{y}}^{-(2+\delta)}\,\dy\\
&\qquad\leq  
\Cc{c}\,
\per^\gamma
%\per^{-\half-\half\bp{-3+\frac{2}{p^2}}} 
\int_{\R^2} 
\snorm{x-y}^{-(1+2\gamma)}
%\snorm{x-y}^{-3+\frac{2}{p^2}}
\,\norm{\rhstpgenerictensor}_{\LR{\infty}\np{\torusper;\wspace{2+\delta}}}
\,\bp{1+\snorm{y}}^{-(2+\delta)}\,\dy.
\end{align*}
%for any $p\in(1,\infty)$. 
We now choose 
$\gamma=\frac{\delta}{2}$  in the estimate above
%$p:=\sqrt{\frac{2}{2-\delta}}$ 
to obtain
\begin{align*}
\snorm{\partial_j\fundsolcomplper_{il}*\rhstpgenerictensor_{lj}\,(t,x)}
&\leq  
\Cc{c}\,\per^{\frac{\delta}{2}}\,\norm{\rhstpgenerictensor}_{\LR{\infty}\np{\torusper;\wspace{2+\delta}}} 
\int_{\R^2} \snorm{x-y}^{-(1+\delta)}
\,\bp{1+\snorm{y}}^{-(2+\delta)}\,\dy\\
&\leq \Cc{c}\,\per^{\frac{\delta}{2}}\, \norm{\rhstpgenerictensor}_{\LR{\infty}\np{\torusper;\wspace{2+\delta}}}\, \bp{1+\snorm{x}}^{-(1+\delta)},
\end{align*}
which implies \eqref{ConvolutionWithPurelyPeriodicFundSolLemma_GradEstimate}.
Estimate
\eqref{ConvolutionWithPurelyPeriodicFundSolLemma_Estimate} 
can be shown in a completely similar manner by using \eqref{FundsolComplPointwiseDecayLem_PointwisedecayEst} with 
$\gamma=\frac{1+\delta}{2}$ and $p\in (1,\frac{2}{1-\delta})$
%$p:=\sqrt{\frac{2}{1-\delta}}$
instead of 
\eqref{FundsolComplPointwiseDecayLem_PointwisedecayEstGrad}.
\end{proof}

\section{Linear Problem} 
\label{linear-problem}

In this section we consider the following linearized system in a rotating frame of reference:
\begin{align}\label{StokesRotFrameWholeSpace}
\begin{pdeq}
&-\Delta\velrotstokes - \angvel(\nsnonlinb{\xrot}{\velrotstokes}-\velrotstokesrot) + \grad\presrotstokes = \rhsrot + \Div\rhsrottensor && \tin\R^2,\\
&\Div\velrotstokes = 0 && \tin \R^2.
\end{pdeq}
\end{align}
Due to efforts of several authors
(\cite{FHM}, \cite{GKy-1}, \cite{GKy-2}, \cite{HMN}, \cite{Hi06}, \cite{Hi}),
we already know the existence of a unique solution under appropriate
conditions on the external force.
Here, we focus on the external force with vanishing torque,
which implies better decay properties of the solution.
This was pointed out first by Hishida \cite[Proposition 5.3.2]{Hi}
through the asymptotic representation of the solution to
\eqref{StokesRotFrameWholeSpace}.
Our task is now to establish a pointwise estimate of the solution
that optimally captures the singular behavior for $a\to 0$.
If in particular $f$ is compactly supported, the singular behavior
$|a|^{-(1+\delta)/2}$ in \eqref{linear-est} below
for $a\to 0$ has been deduced first by Higaki, Maekawa and Nakahara
\cite[Theorem 3.1 (i)]{HMN}.
For the external force of divergence form, the singular behavior
$|a|^{-\delta/2}$ for $a\to 0$ is not explicitly found in
\cite[Theorem 3.1 (ii)]{HMN}, however, it is hidden there.
Note that the following assertion is not valid for $\delta=0$.
\begin{thm}%\label{StokesRotFrameWholeSpaceThm}
Let $\delta\in(0,1)$ and $\angvel \in\mathbb R\setminus\{0\}$.
Suppose that
\[
\rhsrot\in\wspace{3+\delta}(\R^2)^{2}, \qquad
%\quad\text{with}\quad \int_{\R^2}\rhsrot(y)\cdot\yrot\,\dy=0\label{StokesRotFrameWholeSpaceThm_Condition},\\
\rhsrottensor\in\wspace{2+\delta}(\R^2)^{2\times 2}
%\quad\text{with}\quad \rhsrottensor^\transpose=\rhsrottensor\label{StokesRotFrameWholeSpaceThm_ConditionTensor}
\]
and that
\begin{equation}
\int_{\mathbb R^2}y^\perp\cdot f\,\dy
+\int_{\mathbb R^2}(F_{12}-F_{21})\,\dy=0,
\label{torque-free}
\end{equation}
then there is a unique solution $\velrotstokes\in\wspace{1+\delta}(\R^2)^2$ to 
\eqref{StokesRotFrameWholeSpace} which satisfies
\begin{align}%\label{StokesRotFrameWholeSpaceThm_Estimate}
\norm{\velrotstokes}_{\wspace{1+\delta}}\leq \Cc[StokesRotFrameWholeSpaceThmConst]{C} 
\Bp{
\bp{1+|\angvel|^{-\frac{1+\delta}{2}}}\norm{\rhsrot}_{\wspace{3+\delta}}+
\bp{1+|\angvel|^{-\frac{\delta}{2}}}\norm{\rhsrottensor}_{\wspace{2+\delta}}
},
\label{linear-est}
\end{align}
where $\const{StokesRotFrameWholeSpaceThmConst}=\const{StokesRotFrameWholeSpaceThmConst}(\delta)$ is independent of $\angvel$.
\label{linear-thm}
\end{thm}
\begin{proof}
By \cite[Lemma 5.3.5]{Hi} the solution to
\eqref{StokesRotFrameWholeSpace}
is unique within the class of tempered distributions
up to additive (specific) polynomials, and thus unique
within $X_{1+\delta}(\mathbb R^2)^2$.
%It therefore suffices to establish existence of a solution satisfying \eqref{linear-est}.
Since $f\in \wspace{3+\delta}(\mathbb R^2)^2\subset L^p(\mathbb R^2)^2$ and
$F\in \wspace{2+\delta}(\mathbb R^2)^{2\times 2}
\subset L^p(\mathbb R^2)^{2\times 2}$
for every $p\in (1,\infty)$, the argument from
\cite{FHM} and \cite{Hi06} %\cite{GKy-1} and \cite{GKy-2}
yields the existence of a solution $v(x)$ (their argument is valid for 2D
as well, see also \cite{GKy-1} and \cite{GKy-2}).
It can be represented as
the volume potential of $f$ in terms of the associated fundamental solution,
see \cite[Proposition 5.3.2]{Hi} and \cite[Theorem 3.1]{HMN}.
Their analysis of the fundamental solution is very precise, however,
it is convenient to adopt another representation
of the solution $v(x)$ obtained above,
see \eqref{relation}--\eqref{StokesRotFrameWholeSpaceThm_DecompOfTPVel}
below, in order to deduce the desired estimate \eqref{linear-est}.
By \cite[Lemma 2.5]{KMT} one can
decompose $f\in\wspace{3+\delta}(\mathbb R^2)^2$
in the form $f=f_0+\mbox{div $F_0$}$ with
\[
\mbox{Supp $f_0$}\subset B_1(0), \qquad
f_0\in L^\infty(\mathbb R^2)^2, \qquad
F_0\in \wspace{2+\delta}(\mathbb R^2)^{2\times 2},
\]
subject to
\begin{equation}
\begin{split}
&\|f_0\|_\infty\leq C\left|\int_{\mathbb R^2} f(y)\dy\right|
\leq C\|f\|_{\wspace{2+\delta}}, \\
&\|F_0\|_{\wspace{2+\delta}}\leq C\|f\|_{\wspace{3+\delta}},
\end{split}
\label{further-split}
\end{equation}
where $\|\cdot\|_\infty$ denotes the norm of $L^\infty(\mathbb R^2)$.
The external force is then rewritten as
\[
f+\mbox{div $F$}=f_0+\mbox{div $G$}, \qquad G:=F+F_0.
\]
Put $\per:=\frac{2\pi}{|a|}$
%\begin{align*}
%\rotmatrix(t):=
%\begin{pmatrix}
%\cos\angvel t & -\sin\angvel t \\
%\sin\angvel t & \cos\angvel t \\
%\end{pmatrix}
%\end{align*}
and
\begin{align*}
&\rhstpgeneric:\torusper\times\R^2\ra\R^2,\quad 
\rhstpgeneric(t,x)\coloneqq\rotmatrix(t) f_0\bp{\rotmatrix(t)^\transpose x},\\
&\rhstpgenerictensor:\torusper\times\R^2\ra\R^{2\times 2},\quad 
\rhstpgenerictensor(t,x)\coloneqq
\rotmatrix(t) G\bp{\rotmatrix(t)^\transpose x}\rotmatrix(t)^\transpose,
\end{align*}
where the matrix $\rotmatrix(t)$ is given by \eqref{rot-mat}.
Clearly, 
$\rhstpgeneric\in\LR{\infty}\bp{\torusper; L^\infty(\R^2)}^2$
with support being compact in $B_1(0)$ for each $t$ and
$\rhstpgenerictensor\in\LR{\infty}\bp{\torusper;\wspace{2+\delta}(\R^2)}^{2\times 2}$.
One readily verifies (recall \eqref{intro_ConvolutionWithTPFundsol2}) that the distribution $\veltp\in\TDR(\torusper\times\R^2)^2$ defined by $(i=1,2)$
\begin{align*}
\veltp_i:=\fundsoltpvelper_{il}\convolutionpergroup \rhstpgeneric_{l} 
+ \partial_j\fundsoltpvelper_{il}\convolutionpergroup \rhstpgenerictensor_{lj}
\end{align*}
is a solution to 
\begin{align}\label{StokesRotFrameWholeSpaceThm_TPSystem}
\begin{pdeq}
&\partial_t\veltp-\Delta\veltp + \grad\prestp = \rhstpgeneric + \Div\rhstpgenerictensor&& \tin\grp, \\
&\Div\veltp =0 && \tin\grp.
\end{pdeq}
\end{align}
As in the derivation of \eqref{NS} by use of \eqref{trans} together with
the uniqueness for \eqref{StokesRotFrameWholeSpace} mentioned above,
we have the relations
\begin{equation}
w(t,x)=Q(t)v(Q(t)^\top x), \qquad
v(x)=Q(t)^\top w(t,Q(t)x).
\label{relation}
\end{equation}
Recalling \eqref{TimePeriodicFundsolDef}, we find that 
\begin{align}\label{StokesRotFrameWholeSpaceThm_DecompOfTPVel}
\begin{aligned}
\veltp_i&= 
\bp{\fundsolstokesvel\otimes \onedist_{\torus_\per}}_{il}\convolutionpergroup \rhstpgeneric_{l} 
+ \partial_j\bp{\fundsolstokesvel\otimes \onedist_{\torus_\per}}_{il}\convolutionpergroup \rhstpgenerictensor_{lj}
+ \fundsolcomplper_{il}\convolutionpergroup \rhstpgeneric_{l}
+ \partial_j\fundsolcomplper_{il}\convolutionpergroup \rhstpgenerictensor_{lj}\\
&= 
%\fundsolstokesvel_{il}\convolutioneuclidean \rhstpproj_{l} 
%+ {\partial_j\fundsolstokesvel}_{il}\convolutioneuclidean \rhstpprojtensor_{lj}
w^S_i
+ \fundsolcomplper_{il}\convolutionpergroup \rhstpgeneric_{l}
+ \partial_j\fundsolcomplper_{il}\convolutionpergroup \rhstpgenerictensor_{lj},
\end{aligned}
\end{align}
where
\begin{equation}
\label{DefOfWSi}
w^S_i
:=\fundsolstokesvel_{il}\convolutioneuclidean \rhstpproj_{l}
+ {\partial_j\fundsolstokesvel}_{il}\convolutioneuclidean \rhstpprojtensor_{lj}
\end{equation}
with
\begin{align*}
\rhstpproj(x) := \int_{\torusper} \rhstpgeneric(t,x)\,\dt,\quad 
\rhstpprojtensor(x) := \int_{\torusper} \rhstpgenerictensor(t,x)\,\dt,
\end{align*}
which do not depend on $a\in\mathbb R\setminus\{0\}$ because of
\begin{equation*}
\overline{h}(x)
=\frac{1}{\cal T}\int_0^{\cal T}Q(t)f_0(Q(t)^\top x)\,\dt
=\frac{1}{2\pi}\int_0^{2\pi}Q_1(t)f_0(Q_1(t)^\top x)\,\dt
\end{equation*}
and the same description of $\overline{H}(x)$,
where $Q_1(t)$ denotes the matrix \eqref{rot-mat} with $a=1$.
As a consequence, $w^S(x)$ does not depend on the angular velocity $a$.
We compute
\begin{equation}
\int_{\R^2}\rhstpproj(y)\,\dy 
= \int_{\torusper} \int_{\R^2} \rotmatrix(t)\rhsrot_0\bp{\rotmatrix(t)^\transpose y}\,\dy\dt   
= \int_{\torusper} \rotmatrix(t) \int_{\R^2} \rhsrot_0\bp{ y}\,\dy\dt = 0  
\label{1st-coe}
\end{equation}
and 
\begin{equation*}
\begin{split}
\int_{\R^2}\rhstpproj(y)\otimes y \,\dy 
&=\int_{\torusper} \int_{\R^2} \bb{\rotmatrix(t)\rhsrot_0(y)}\otimes\bb{\rotmatrix(t) y}\,\dy\dt  \\
&=\half\int_{\mathbb R^2}y\cdot \rhsrot_0\,\dy %\idmatrix
\begin{pmatrix}
1 & 0 \\
0 & 1
\end{pmatrix}
+\half\int_{\mathbb R^2}y^\perp\cdot \rhsrot_0\,\dy  
\begin{pmatrix}
0 & -1\\
1 & 0
\end{pmatrix}.
\end{split}
\end{equation*}
Moreover, letting $\intmatrix{G}:=\int_{\R^2}G(y)\,\dy$,
we find that 
\begin{equation*}
\begin{split}
\int_{\R^2}\rhstpprojtensor(y)\,\dy 
&= \int_{\torusper} \rotmatrix(t)\intmatrix{G}\rotmatrix(t)^\transpose\,\dt
= \half
\begin{pmatrix}
\trace\intmatrix{G} & \intmatrix{G}_{12}-\intmatrix{G}_{21}\\
\intmatrix{G}_{21}-\intmatrix{G}_{12} & \trace\intmatrix{G}
\end{pmatrix}   \\
&=\half\trace\intmatrix{G} %\idmatrix
\begin{pmatrix}
1 & 0 \\
0 & 1
\end{pmatrix}
-\half (\intmatrix{G}_{12}-\intmatrix{G}_{21})
\begin{pmatrix}
0 & -1 \\
1 & 0
\end{pmatrix}\\
&=\half\trace\intmatrix{G} %\idmatrix
\begin{pmatrix}
1 & 0 \\
0 & 1
\end{pmatrix}
-\half \Bp{
 (\intmatrix{F}_{12}-\intmatrix{F}_{21})
+\int_{\R^2}y^\perp\cdot \Div F_0\,\dy
}
\begin{pmatrix}
0 & -1 \\
1 & 0
\end{pmatrix}.
\end{split}
\end{equation*}
Due to the Stokes fundamental solution being solenoidal and by the assumption
\eqref{torque-free},
the computations above imply that
\begin{align*}
\left\{-\Bb{\int_{\R^2}\overline{h}(y)\otimes y\,\dy}_{lj}+
\Bb{\int_{\R^2}\rhstpprojtensor(y)\,\dy}_{lj}\right\}
\partial_j\fundsolstokesvel_{il}(x) = 0,
%\half\trace\intmatrix{F}\partial_j\fundsolstokesvel_{ij}(x) = 0. 
\end{align*}
which together with \eqref{1st-coe}
as well as Lemma \ref{SteadyStokesRepresentationFundsol} leads to
\begin{equation}
\sup_{|x|\geq e}|x|^{1+\delta}|w^S(x)|
\leq C(\|\overline{h}\|_{\wspacelog{3+\delta}}
+\|\overline{H}\|_{X_{2+\delta}})
\leq C(\|f_0\|_\infty+\|G\|_{X_{2+\delta}}).
\label{w-steady}
\end{equation}
On the other hand, it is easily derived from \eqref{DefOfWSi} and basic estimates of $\fundsolstokesvel$ that  
\[
\sup_{|x|<e}|w^S(x)|
\leq C(\|f_0\|_\infty+\|G\|_{X_{2+\delta}}).
%\label{local-est} 
\]
Combining this with \eqref{w-steady}, we find
\begin{equation}
\|w^S\|_{X_{1+\delta}}
\leq C(\|f_0\|_\infty+\|G\|_{X_{2+\delta}}),
\label{steady-est} 
\end{equation}
where the constant $C>0$ is independent of the angular velocity
$a\in\mathbb R\setminus\{0\}$
since so is $w^S(x)$ as mentioned above.
%We claim that the constant $C$ in this estimate is independent on $a$. Indeed,a change of variables yields
%\begin{align*}
%\rhstpproj(x) =
%\int_{\torusper} \rotmatrix(t)\rhsrot\bp{\rotmatrix(t)^\transpose x} \,\dt =
%\frac{1}{2\pi}\int_0^{2\pi} \rotmatrix_1(t)\rhsrot\bp{\rotmatrix_1(t)^\transpose x}
%\end{align*}
%where $\rotmatrix_1(t)$ denotes the matrix in \eqref{rot-mat} with $a=1$. This means that $\rhstpproj$ is independent on $\per$, and consequently also on $a$. By a similar argument, the same is true for $\rhstpprojtensor$, and by \eqref{DefOfWSi} thus also for $w^S$.
%We conclude that the constant $C$ in \eqref{steady-est} is independent on $a$.

Returning to \eqref{StokesRotFrameWholeSpaceThm_DecompOfTPVel}, we collect
%Lemma \ref{SteadyStokesRepresentationFundsol} 
\eqref{further-split}, \eqref{steady-est} and
Lemma \ref{ConvolutionWithPurelyPeriodicFundSolLemma} to conclude that 
\begin{equation*}
\begin{split}
\norm{\veltp}_{\LR{\infty}\np{\torusper;\wspace{1+\delta}}}
&\leq C
\Bp{
\bp{1+|\angvel|^{-\frac{1+\delta}{2}}}\|f_0\|_\infty+
\bp{1+|\angvel|^{-\frac{\delta}{2}}}\|F+F_0\|_{\wspace{2+\delta}}
} \\
&\leq C
\Bp{
\bp{1+|\angvel|^{-\frac{1+\delta}{2}}}\norm{\rhsrot}_{\wspace{3+\delta}}+
\bp{1+|\angvel|^{-\frac{\delta}{2}}}\norm{\rhsrottensor}_{\wspace{2+\delta}}
},
\end{split}
%\label{StokesRotFrameWholeSpaceThm_EstOfveltp}
\end{equation*}
which leads to \eqref{linear-est} in view of \eqref{relation}.
\end{proof}

%As in the derivation of \eqref{NS} by use of \eqref{trans}, a direct computation shows that  
%\begin{align*}
%\velrotstokes(x):= \int_{\torusper} \rotmatrix(t)^\transpose \veltp\bp{t,\rotmatrix(t)x}\,\dt
%\end{align*}
%is a solution to \eqref{StokesRotFrameWholeSpace}. 
%\eqref{StokesRotFrameWholeSpaceThm_Estimate}.

%It remains to show uniqueness. Assume for this purpose that $\velrotstokes\in\wspace{1+\delta}(\R^2)^2$ is a solution to \eqref{StokesRotFrameWholeSpace} with $f=F=0$. Then 
%\begin{align*}
%\veltp:\torusper\times\R^2\ra\R^2,\quad \veltp(t,x):=\rotmatrix(t)\velrotstokes\bp{\rotmatrix(t)^\transpose x}
%\end{align*}
%is a solution $\veltp\in\LR{\infty}\bp{\torusper;\wspace{1+\delta}(\R^2)}\subset\TDR(\grpper)$ to \eqref{StokesRotFrameWholeSpaceThm_TPSystem} with $\rhstpgeneric=\rhstpgenerictensor=0$. Applying the Fourier transform
%$\FT_{\grpper}$ in \eqref{StokesRotFrameWholeSpaceThm_TPSystem}, one can thus deduce $\supp\FT_{\grpper}\np{\veltp} \subset \set{(0,0)}$. This implies that $\veltp$ is a (time-independent) polynomial in $x$. Since $\veltp\in\LR{\infty}\bp{\torusper;\wspace{1+\delta}(\R^2)}$, this is only possible if $\veltp=0$. 

\section{Proof of Theorem \ref{main}}
\label{proof-main}
\begin{proof}
Let us fix $\psi\in C^\infty([0,\infty))$ such that
$\psi(\rho)=1$ for $0\leq\rho\leq 4/3$ and
$\psi(\rho)=0$ for $\rho\geq 5/3$.
Given $R\in (e,\infty)$ satisfying
$\mathbb R^2\setminus\Omega\subset B_R(0)$,
we set $\varphi(x)\coloneqq\psi(|x|/R)$ for $x\in\mathbb R^2$. By
$\mathbb B$ we denote the Bogovskii operator that associates a function $h$ with vanishing mean value over $A$ with a particular
solution $H$ constructed by Bogovskii \cite{Bog}
(see also \cite{BS} and \cite[Chapter III.3]{G-b}) to the boundary value problem for the
divergence equation $\Div H=h$ in the annulus
$$A:=\setc{x\in\R^2}{R<\snorm{x}<2R}$$
subject to the homogeneous Dirichlet boundary condition, \textit{i.e.}, $\Div\mathbb B\nb{h}=h$ in $A$ and
$\mathbb B\nb{h}=0$ on $\partial A$ provided $\int_A h\,\dx=0$.

Given a solution $\{u,p\}$ to \eqref{NS} satisfying \eqref{zero-flux} and decaying like $u(x)=O(|x|^{-1})$ as $|x|\to\infty$,
we set
\begin{equation*}
\begin{split}
&\widetilde u \coloneqq(1-\varphi)u+\mathbb B[u\cdot\nabla\varphi], \qquad
\widetilde p \coloneqq(1-\varphi)p, \\
&\widetilde U \coloneqq(1-\varphi)U+\mathbb B[U\cdot\nabla\varphi], \qquad
\widetilde P \coloneqq(1-\varphi)P, 
\end{split}
\end{equation*}
where
\begin{align}
U(x) \coloneqq \frac{M}{4\pi} \frac{x^\perp}{|x|^2}, \qquad
P(x) \coloneqq \Bp{\frac{M}{4\pi}}^2 \frac{-1}{2|x|^2} 
\label{candidate}
\end{align}
is the flow discussed in Section \ref{intro}, see \eqref{circu},
and the constant $M$ is defined by \eqref{torque}.
The reason of this choice $c \coloneqq M/(4\pi)$ in \eqref{circu} is clarified later.
Note that
$\int_A u\cdot\nabla\varphi\, dx=0$ follows from \eqref{zero-flux},
while a direct computation yields
$$\int_A U\cdot\nabla\varphi\,dx= \int_A \Div\bb{\varphi U}\,\dx = 
\int_{|x|=R}\frac{-x}{R}\cdot U\,d\sigma=0.$$
By well-known estimates of the Bogovskii operator (see \cite{Bog},
\cite{BS} and \cite[Chapter III.3]{G-b}),
we have
$\widetilde u,\, \widetilde U\in X_{1}$ with
\begin{equation}
\|\widetilde u\|_{X_1}\leq C\sup_{|x|\geq R}|x||u(x)|, \qquad
\|\widetilde U\|_{X_1}\leq C|M|,
\label{est-tilde} 
\end{equation}
where $C>0$ is independent of $R$ on account of
dilation invariance of the constant in the $L^q$-estimate
due to \cite[Theorem 2.10]{BS} (see also \cite[Theorem III.3.1]{G-b}).

The pair 
\[
v:=\widetilde u-\widetilde U, \qquad
\pi:=\widetilde p-\widetilde P
\]
obeys
\begin{equation}
-\Delta v-a(x^\perp\cdot\nabla v-v^\perp)+\nabla\pi
=(1-\varphi)f+g+\mbox{div $J$}(v), \quad
\mbox{div $v$}=0
\label{NS2}
\end{equation}
in $\mathbb R^2$, where
\[
J(v)\coloneqq-(\widetilde u\otimes v+v\otimes\widetilde u)+v\otimes v
\]
and
\[
g \coloneqq h(u,p)-h(U,P)
\]
with
\begin{equation*}
\begin{split}
h(u,p)
& \coloneqq2\nabla\varphi\cdot\nabla u
+(\Delta\varphi +ax^\perp\cdot\nabla\varphi)u
-\Delta\mathbb B[u\cdot\nabla\varphi]
-ax^\perp\cdot\nabla\mathbb B[u\cdot\nabla\varphi]  \\
&\quad +a\mathbb B[u\cdot\nabla\varphi]^\perp
-(\nabla\varphi)p 
+(1-\varphi)u\cdot\nabla
\left\{-\varphi u+\mathbb B[u\cdot\nabla\varphi]\right\}  \\
&\quad +\mathbb B[u\cdot\nabla\varphi]
\cdot\nabla\big\{(1-\varphi)u+\mathbb B[u\cdot\nabla\varphi]\big\}.
\end{split}
\end{equation*}
It is seen that $g\in C_0^\infty(A)$ and
\begin{equation} 
\sup_{x\in A}|g(x)| 
\leq c(R)\,(1+|a|)(|M|+N)
\label{est-g}
\end{equation}
with
\begin{align*}
N:= {\sup_{R<|x|<2R} \bp{\snorm{u(x)} + \snorm{\nabla u(x)} + \snorm{\nabla^2 u(x)} + \snorm{p(x)}}}
\end{align*}
for some constant $c(R)>0$,
which depends on $R$ but is independent of $a$,
provided both $N$ and $|M|$ are smaller than one 
(by choosing $\mu$ small enough in \eqref{small}) so that
$M^2\leq |M|$, $N^2\leq N$.

By essentially the same computation as in \cite[Section 5.4]{Hi}
we deduce
\begin{equation*}
\begin{split}
&\quad \int_{\mathbb R^2}y^\perp\cdot\{(1-\varphi)f+h(u,p)\}\,\dy \\
&=\int_{\mathbb R^2}y^\perp\cdot\{
-\Delta\widetilde u-a(y^\perp\cdot\nabla\widetilde u-\widetilde u^\perp)
+\nabla\widetilde p+\widetilde u\cdot\nabla\widetilde u\}\,\dy
=M,
\end{split}
\end{equation*}
where $M$ is given by \eqref{torque}.
Similarly, we find
\[
\int_{\mathbb R^2}y^\perp\cdot h(U,P)\,\dy
=\int_{|y|=1}y^\perp\cdot\{(T(U,P)+a\,U\otimes y^\perp-U\otimes U)\nu\}\,d\sigma,
\]
which we compute with use of the explicit form \eqref{circu} to show that
the last boundary integral is equal to $4\pi c$, where $c$ is the parameter
in \eqref{circu}.
We thus choose $c \coloneqq M/(4\pi)$ as in \eqref{candidate} to obtain
\[
\int_{\mathbb R^2}y^\perp\cdot\{(1-\varphi)f+g\}\,dy=0,
\]
which combined with symmetry $J_{12}(v)=J_{21}(v)$
enables us to reconstruct a solution
$V\in X_{1+\delta}(\mathbb R^2)^2$ to \eqref{NS2} satisfying
\begin{equation}
\|V\|_{X_{1+\delta}}\leq
L:=2\const{StokesRotFrameWholeSpaceThmConst}\,\bp{1+|a|^{-(1+\delta)/2}}\,\|(1-\varphi)f+g\|_{\wspace{3+\delta}}
\label{est-V}
\end{equation}
(together with the associated pressure $\Pi$)
under the smallness conditions \eqref{small}
by means of a fixed point argument based on Theorem \ref{linear-thm},
where $\const{StokesRotFrameWholeSpaceThmConst}=\const{StokesRotFrameWholeSpaceThmConst}(\delta)$ is as in this theorem.
In fact,
\begin{equation}
\begin{split}
L&\leq
C(1+|a|^{-(1+\delta)/2})\sup_{|x|\geq R}|x|^{3+\delta}|f(x)|  \\
&\quad +Cc(R)R^{3+\delta}(1+|a|^{-(1+\delta)/2})(1+|a|)(|M|+N)
\end{split}
\label{est-L} 
\end{equation}
follows from \eqref{est-g} and, thereby,
the conditions \eqref{small} with appropriate constants
$\kappa=\kappa(\delta)$, $\mu=\mu(\delta,R)$ imply that
$(1+|a|^{-\delta/2})L$ is sufficiently small;
as a consequence, we obtain a solution
$V\in X_{1+\delta}(\mathbb R^2)^2$ with \eqref{est-V}.
%Consequently, a fixed point argument based on \eqref{linear-est} yields a solution to \eqref{NS2} in the ball of radius $L$ in the space $X_{1+\delta}$.

We now show that $V$ constructed above coincides with
$v=\widetilde u-\widetilde U$.
To this end, put
\[
w:=v-V, \qquad \sigma:=\pi-\Pi,
\]
which obey
\[-\Delta w-a(x^\perp\cdot\nabla w-w^\perp)+\nabla\sigma=\mbox{div $K$}(w),
\qquad\mbox{div $w$}=0 
\]
in $\mathbb R^2$ with
\[
K(w) \coloneqq-(\widetilde u\otimes w+w\otimes\widetilde u)+v\otimes w+w\otimes V.
\]
%At this point, it would be beneficial to employ \eqref{linear-est} to estimate $w$. However,
At the outset $K(w)$ only belongs to $X_2(\R^2)^{2\times 2}$, so Theorem \ref{linear-thm} is not applicable because the case $\delta=0$ is not admissible.
We therefore rely on the $L^q$-theory; indeed,
$K(w)\in L^q(\mathbb R^2)^{2\times 2}$ for every $q\in (1,\infty)$.
Let us fix $q\in (1,2)$.
The a priori estimate obtained in
\cite{Hi06} and \cite{GKy-2}
(where 3D case is discussed, but the argument is similar for 2D)
together with the Sobolev embedding relation yields
\begin{equation*}
\|w\|_{q_*,q}
\leq C\|\nabla w\|_q
\leq C\|K(w)\|_q   
\end{equation*}
with a constant $C$ independent of $a$,
since a simple scaling argument implies that the constant in the 
$L^q$-estimate for \eqref{linear} in $\mathbb R^2$ does not depend on $a$.
Here,
$\|\cdot\|_{q_*,q}$ with $q_*=2q/(2-q)$ and $\|\cdot\|_{2,\infty}$
denote the norms of the Lorentz spaces
$L^{q_*,q}(\mathbb R^2)$ and $L^{2,\infty}(\mathbb R^2)$, respectively,
while the $L^q$-norm is denoted by $\|\cdot\|_q$.
Further employing the Lorentz-H\"older inequality, we obtain
\begin{align*}
\|w\|_{q_*,q} &\leq C\big(\|\widetilde u\|_{2,\infty}+\|v\|_{2,\infty}
+\|V\|_{2,\infty}\big)\|w\|_{q_*,q}  \\ 
&\leq C\big(\|\widetilde u\|_{X_1}
+\|\widetilde U\|_{X_1}
+\|V\|_{X_{1+\delta}}\big)\|w\|_{q_*,q}.
\end{align*}
We thus conclude that $v=V$, yielding \eqref{structure},
whenever
$\|\widetilde u\|_{X_1}+\|\widetilde U\|_{X_1}+\|V\|_{X_{1+\delta}}$
is small enough.
This latter condition can be accomplished by
\eqref{small} (with still smaller $\kappa,\, \mu$)
on account of \eqref{est-tilde}, \eqref{est-V} and \eqref{est-L}.
\end{proof}

%%%%%%%%%%%%%%%%%%%%%%%%%%%%%%%%%%%%%%%%%%%%%%%%%%%%%%%%%%%%%%
%%          Acknowledgment                                  %%
%%%%%%%%%%%%%%%%%%%%%%%%%%%%%%%%%%%%%%%%%%%%%%%%%%%%%%%%%%%%%%

%%%%%%%%%%%%%%%%%%%%%%%%%%%%%%%%%%%%%%%%%%%%%%%%%%%%%%%%%%%%%%
%%          Bibliography                                    %%
%%%%%%%%%%%%%%%%%%%%%%%%%%%%%%%%%%%%%%%%%%%%%%%%%%%%%%%%%%%%%%
%\bibliographystyle{abbrv}
%\bibliography{assfar2Db.bib}

%%%%%%%%%%%%%%%%%%%%%%%%%%%%%%%%%%%%%%%%%%%%%%%%%%%%%%%%%%%%%
%          Questions and answers                           %%
%%%%%%%%%%%%%%%%%%%%%%%%%%%%%%%%%%%%%%%%%%%%%%%%%%%%%%%%%%%%%
% \section{Q\&A}
% \begin{enumerate}
% \item end{itemize}
% \end{enumerate} 

\end{document}